\pdfoutput=1 		% For the arXiv
\documentclass[leqno]{article}
\usepackage[normal]{phaine_style}
\usepackage[margin=1.5in]{geometry}
\theoremstyle{definition}
\newtheorem{proof_strategy}[equation]{Proof Strategy}
\newtheorem{linear_overview}[equation]{Linear Overview}
\newtheorem{conventions}[equation]{Conventions}

\newcommand{\operad}{\xspace{$\infty$-op\-er\-ad}\xspace}
\newcommand{\operads}{\xspace{$\infty$-op\-er\-ads}\xspace}

\newcommand{\procategories}{\xspace{pro-$\infty$-cat\-e\-gories}\xspace}

\newcommand{\Sigmaprofinite}{\xspace{$\Sigma$-pro\-fi\-nite}\xspace}
\newcommand{\Sigmacompletion}{\xspace{$\Sigma$-com\-ple\-tion}\xspace}

\newcommand{\Bprofin}{\Bup\profincomp}
\newcommand{\Btrun}{\Bup_{<\infty}}

\newcommand{\protrun}{\trun_{<\infty}}

\newcommand{\profincomp}{_{\uppi}^{\wedge}}

\renewcommand{\Pihat}{\widehat{\Pi}}

\newcommand{\Piet}{\Pi_{\infty}^{\et}}
\newcommand{\Pietprotrun}{\Pi_{<\infty}^{\et}}

\DeclareMathOperator{\qcqs}{qcqs}
\newcommand{\Schqcqs}{\Sch^{\qcqs}}

\newcommand{\Xbullet}{X_{\bullet}}
\newcommand{\Ybullet}{Y_{\bullet}}

\newcommand{\Shapeprotrun}{\Pi_{<\infty}}

\renewcommand{\Setfin}{\Set_{\fin}}

\newcommand{\Catfin}{\Cat_{\infty,\uppi}}
\newcommand{\ProCatfin}{\Pro(\Catfin)}

\DeclareMathOperator{\Catob}{Cat}
\newcommand{\Cattrun}{\Cat_{<\infty}}
\newcommand{\ProCattrun}{\Pro(\Cattrun)}

\newcommand{\ProSpc}{\Pro(\Spc)}
\newcommand{\Spcfin}{\Space_{\uppi}}
\newcommand{\Spctrun}{\Space_{<\infty}}
\newcommand{\ProSpcfin}{\Pro(\Spcfin)}
\newcommand{\ProSpctrun}{\Pro(\Spctrun)}

\newcommand{\SpcSigma}{\Spc_{\Sigma}}
\newcommand{\ProSpcSigma}{\Pro(\SpcSigma)}
\newcommand{\Sigmacomp}{_{\Sigma}^{\wedge}}

\title{\Large Profinite completions of products}

\author{\normalsize Peter J. Haine}

\date{\normalsize \today}

\begin{document}

\maketitle

%-------------------------------------------------------------------%
% Abstract                                                          %
%-------------------------------------------------------------------%

\begin{abstract} 
	A source of difficulty in profinite homotopy theory is that the profinite completion functor does not preserve finite products.
	In this note, we provide a new, checkable criterion on prospaces $ X $ and $ Y $ that guarantees that the profinite completion of $ X \cross Y $ agrees with the product of the profinite completions of $ X $ and $ Y $.
	Using this criterion, we show that profinite completion preserves products of étale homotopy types of qcqs schemes.
	This fills a gap in Chough's proof of the Künneth formula for the étale homotopy type of a product of proper schemes over a separably closed field.
\end{abstract}

\tableofcontents

%-------------------------------------------------------------------%
%-------------------------------------------------------------------%
%  Introduction                                                     %
%-------------------------------------------------------------------%
%-------------------------------------------------------------------%

\setcounter{section}{-1}

\section{Introduction}

Write $ \Spcfin $ for the \category of \pifinite spaces.
Given a set $ \Sigma $ of primes, write $ \SpcSigma \subset \Spcfin $ for the full subcategory spanned by those \pifinite spaces whose homotopy groups have orders divisible only by primes in $ \Sigma $.
Write 
\begin{equation*}
	(-)\Sigmacomp \colon \fromto{\ProSpc}{\ProSpcSigma}
\end{equation*}
for the \Sigmacompletion functor, i.e., the left adjoint to the inclusion $ \ProSpcSigma \subset \ProSpc $.
% Note that if $ \Sigma $ consists of a single prime $ p $, then \Sigmacompletion is $ p $-completion, and if $ \Sigma $ is the set of all primes then \Sigmacompletion is profinite completion.
One source of difficulty in profinite homotopy theory is that the \Sigmacompletion functor does not preserve finite limits, or even finite products (see \cites[\SAGthm{Remark}{E.5.2.6}]{SAG}[Remark 3.10]{MR3921321}).

As far as we are aware, given connected spaces $ X $ and $ Y $, the only general condition to check that the profinite completion of $ X \cross Y $ is the product of the profinite completions of $ X $ and $ Y $ is to check that the homotopy groups of $ X $ and $ Y $ are \textit{good} in the sense of Serre \cite[Proposition 3.9]{MR3921321}.
However, it is generally quite difficult to check if a group is good, and there are hard conjectures about whether or not certain groups are good.
For example, Deligne and Morava's conjecture that the mapping class group $ \Gamma_{g,n} $ of a genus $ g $ curve with $ n $ marked points is good \cite[Problem on p. 94]{MR1483111} is still open.

The purpose of this note is to provide a new, checkable criterion on prospaces $ X $ and $ Y $ that guarantees that the natural map $ \fromto{(X \cross Y)\Sigmacomp}{X\Sigmacomp \cross Y\Sigmacomp} $ is an equivalence.
Since the statement of the this criterion requires introducing a bit of terminology, in this introduction we state the two applications that motivated our general results.
For the precise statements of the general results, see \Cref{cor:localization_preserves_products_of_objects_with_iterated_resolutions,thm:profinite_completion_preserves_products_admitting_resolutions_by_simplicial_profinite_spaces}.

The first application is that if one is already in the setting of profinite homotopy theory, then \Sigmacompletion preserves products:

\begin{proposition}[(\Cref{cor:Sigma-completion_preserves_products})]\label{intro_prop:Sigma-completion_preserves_products}
	Let $ \Sigma $ be a set of primes.
	Then the $ \Sigma $-completion functor restricted to profinite spaces
	\begin{equation*}
		(-)\Sigmacomp \colon \fromto{\ProSpcfin}{\ProSpcSigma}
	\end{equation*}
	preserves products.
\end{proposition}

The second is that in the setting of étale homotopy theory, \Sigmacompletion preserves finite products.
Given a scheme $ X $, write $ \Piet(X) \in \ProSpc $ for the étale homotopy type of $ X $.

\begin{proposition}[(\Cref{ex:profinite_completion_preserves_products_of_etale_homotopy_types})]\label{intro_prop:completions_of_products_of_etale_homotopy_types}
	Let $ \Sigma $ be a set of primes and let $ X $ and $ Y $ be qcqs schemes.
	Then the natural map of profinite spaces
	\begin{equation*}
		\paren{\Piet(X) \cross \Piet(Y)}\Sigmacomp \to \Piet(X)\Sigmacomp \cross \Piet(Y)\Sigmacomp
	\end{equation*}
	is an equivalence.
\end{proposition}

\begin{remark}
	\Cref{intro_prop:completions_of_products_of_etale_homotopy_types} fills a gap in Chough's proof of the Künneth formula for the étale homotopy type of a product of proper schemes over a separably closed field \cite[Theorem 5.3]{MR4493612}.
	Chough's proof cites the false claim that profinite completion preserves finite limits.
	However, what Chough actually uses is \Cref{intro_prop:completions_of_products_of_etale_homotopy_types} (with $ \Sigma $ the set of all primes).
	In particular, the conclusion of \cite[Theorem 5.3]{MR4493612} remains valid.

	We also remark that in our work with Holzschuh and Wolf \cite[\S4]{arXiv:2304.00938}, \Cref{intro_prop:completions_of_products_of_etale_homotopy_types} is a key ingredient used to prove other Künneth formulas in étale homotopy theory.
\end{remark}

\begin{proof_strategy}
	\Cref{intro_prop:Sigma-completion_preserves_products,intro_prop:completions_of_products_of_etale_homotopy_types} are consequences of a more general result.
	To explain why this is the case, first note that since \Sigmacompletion preserves cofiltered limits, to prove \Cref{intro_prop:Sigma-completion_preserves_products} it suffices to show that \Sigmacompletion preserves finite products of \pifinite spaces.
	This reduction is useful because \pifinite spaces admit very nice presentations: every \pifinite space can be written as the geometric realization (in $ \Spc $) of a Kan complex with finitely many simplices in each dimension \SAG{Lemma}{E.1.6.5}.

	Similarly, to prove \Cref{intro_prop:completions_of_products_of_etale_homotopy_types} we use that the étale homotopy type of a qcqs scheme admits a nice presentation.
	To see this, the first technical observation is that since protruncation preserves limits \cite[Proposition 3.9]{arXiv:2209.03476} and profinite completion factors through protruncation, it suffices to replace the étale homotopy types by their protruncations.
	Our work with Barwick and Glasman \cite[Theorems 10.2.3 \& 12.5.1]{arXiv:1807.03281} provides a description of the protruncated étale homotopy type as the protruncated classifying space of an explicit profinite category.
	Said differently, the protruncated étale homotopy type can be written as a geometric realization of a simplicial profinite space computed in the larger \category of protruncated spaces (see \Cref{ex:protruncated_shapes_of_spectral_topoi_admit_profinite_resolutions}).

	Hence we're done if we can prove the more general claim that \Sigmacompletion preserves products of protruncated spaces that admit such presentations; this is our main result, see \Cref{thm:profinite_completion_preserves_products_admitting_resolutions_by_simplicial_profinite_spaces}.
	This follows once we know that that geometric realizations preserve finite products in the \categories of protruncated and \Sigmaprofinite spaces (see \Cref{prop:almost_finite_colimits_in_protruncated_spaces_are_universal,cor:almost_finite_colimits_in_pro-L-finite_spaces_are_universal}).
	See \Cref{lem:localization_preserves_products_of_objects_with_resolutions,cor:localization_preserves_products_of_objects_with_iterated_resolutions} for the key categorical results that we use.
\end{proof_strategy}

\begin{linear_overview}
	\Cref{sec:universality_of_colimits} proves that geometric realizations are universal in the \categories of protruncated and profinite spaces.
	In particular, geometric realizations preserve finite products in these \categories.
	\Cref{sec:completions_of_products} proves \Cref{intro_prop:Sigma-completion_preserves_products,intro_prop:completions_of_products_of_etale_homotopy_types}.
	It is immediate from \cite[Theorem 10.2.3]{arXiv:1807.03281} that the protruncated étale homotopy type can be written as the geometric realization of a simplicial profinite space.
	However, for ease of reference we have provided a detailed explanation of this fact in \Cref{app:classifying_prospaces_via_geometric realizations}.
\end{linear_overview}

\begin{conventions}
	Throughout, we use the notational conventions of \cite[\S\S1 \& 3]{arXiv:2209.03476}.
	In an effort to keep this note short, we do not recapitulate them here.
\end{conventions}

%-------------------------------------------------------------------%
%  Acknowledgments                                                  %
%-------------------------------------------------------------------%

\begin{acknowledgments}
	We thank Luciana Basualdo Bonatto, Chang-Yeon Chough, Tim Holz\-schuh, Marcy Robertson, and Sebastian Wolf for many enlightening discussions around the contents of this note.
	We gratefully acknowledge support from the NSF Mathematical Sciences Postdoctoral Research Fellowship under Grant \#DMS-2102957 and a grant from the Simons Foundation (816048, LC). 
\end{acknowledgments}

%-------------------------------------------------------------------%
%-------------------------------------------------------------------%
%  Universality of colimits                                         %
%-------------------------------------------------------------------%
%-------------------------------------------------------------------%

\section{Universality of colimits}\label{sec:universality_of_colimits}

In this section, we prove that geometric realizations are universal in the \categories of protruncated and \Sigmaprofinite spaces.
We accomplish this by proving a more general fact: colimits over diagrams that can be computed as finite colimits when valued in an $ n $-category (see \Cref{def:almost_finite}) are universal in the \categories of protruncated and \Sigmaprofinite spaces (\Cref{prop:almost_finite_colimits_in_protruncated_spaces_are_universal,cor:almost_finite_colimits_in_pro-L-finite_spaces_are_universal}).

The first observation is that finite colimits are universal in protruncated spaces.

\begin{lemma}\label{lem:finite_colimits_universal_in_ProC}
	Let $ \Ccal $ be \acategory with pullbacks and finite colimits.
	If finite colimits are universal in $ \Ccal $, then finite colimits are universal in $ \Pro(\Ccal) $.
\end{lemma}

\begin{proof}
	By (the dual of) \HTT{Proposition}{5.3.5.15}, pullbacks, pushouts, and finite coproducts are computed `levelwise' in $ \Pro(\Ccal) $.
	Thus the assumption that finite colimits are universal in $ \Ccal $ implies the claim.
\end{proof}

\begin{example}\label{ex:finite_colimits_are_universal_in_ProSpc}
	Finite colimits are universal in $ \Pro(\Spc) $.
	For each integer $ n \geq 0 $, finite colimits are universal in $ \Pro(\Spc_{\leq n}) $.
\end{example}

\begin{recollection}
	A localization $ L \colon \fromto{\Ccal}{\Dcal} $ is \defn{locally cartesian} if for any cospan $ X \to Z \ot Y $ such that $ X,Z \in \Dcal $, the natural map $ \fromto{L(X \cross_Z Y)}{X \cross_Z L(Y)} $ is an equivalence.
\end{recollection}

\begin{example}[{\cite[Proposition 3.18]{arXiv:2209.03476}}]
	For any set $ \Sigma $ of primes, the localization
	\begin{equation*}
		(-)\Sigmacomp \colon \fromto{\ProSpctrun}{\ProSpcSigma}
	\end{equation*}
	is locally cartesian.
	However, $ (-)\Sigmacomp $ does not generally preserve finite products.
\end{example}

\noindent The following is immediate from the definitions:

\begin{lemma}\label{lem:locally_cartesian_localizations_preserve_universality_of_colimits}
	Let $ \Ical $ be \acategory, $ \Ccal $ \acategory with pullbacks and $ \Ical $-shaped colimits, and let $ L \colon \fromto{\Ccal}{\Dcal} $ be a locally cartesian localization.
	If $ \Ical $-shaped colimits are universal in $ \Ccal $, then $ \Ical $-shaped colimits are universal in $ \Dcal $.
\end{lemma}

\begin{example}\label{ex:finite_colimits_in_ProSpctrun_are_universal}
	Since the protruncation functor $ \protrun \colon \fromto{\ProSpc}{\ProSpctrun} $ preserves limits \cite[Proposition 3.9]{arXiv:2209.03476}, \Cref{ex:finite_colimits_are_universal_in_ProSpc,lem:locally_cartesian_localizations_preserve_universality_of_colimits} show that finite colimits are universal in $ \ProSpctrun $.
\end{example}

Now we formulate the key property of the category $ \Deltaop $ that we need.

\begin{definition}\label{def:n-colimit-cofinal}
	Let $ n \geq 0 $ be an integer.
	A functor between \categories $ c \colon \fromto{\Ical}{\Jcal} $ is \defn{$ n $-colimit-cofinal} if for every $ n $-category $ \Ccal $ and functor $ f \colon \fromto{\Jcal}{\Ccal} $, the following conditions are satisfied: 
	\begin{enumerate}[label=\stlabel{def:n-colimit-cofinal}]
		\item The colimit $ \colim_{\Jcal} f $ exists if and only if the colimit $ \colim_{\Ical} fc $ exists.

		\item If the colimit $ \colim_{\Jcal} f $ exists, then the natural map $ \fromto{\colim_{\Ical} fc}{\colim_{\Jcal} f} $ is an equivalence.
	\end{enumerate}
\end{definition}
	
\begin{example}\label{ex:Deltaleqn_n-colimit_cofinal}
	For an integer $ n \geq 0 $, write $ \DDelta_{\leq n} \subset \DDelta $ for the full subcategory spanned by those nonempty linearly ordered finite sets of cardinality $ \leq n + 1 $.
	By \cite[Proposition A.1]{arXiv:2207.09256}, the inclusion \smash{$ \Deltaop_{\leq n} \subset \Deltaop $} is $ n $-colimit-cofinal.
\end{example}

\begin{definition}\label{def:almost_finite}
	Let $ \Ical $ be \acategory.
	We say that $ \Ical $ is \defn{almost finite} if for each integer $ n \geq 0 $, there exists a \textit{finite} \category $ \Ical_n $ and an $ n $-colimit-cofinal functor $ c_n \colon \fromto{\Ical_n}{\Ical} $.
\end{definition}

Here are a number of important examples of almost finite \categories.

\begin{example}
	If $ \Ical $ is \acategory that admits a colimit-cofinal functor from a finite \category, then $ \Ical $ is almost finite.
\end{example}

\begin{example}
	For each $ n \geq 0 $, the category $ \Deltaop_{\leq n} $ is a finite \category \cite[Example 6.5.3]{arXiv:2009.07223}.
	Hence the category $ \Deltaop $ is almost finite: the inclusion \smash{$ \incto{\Deltaop_{\leq n}}{\Deltaop} $} is an $ n $-colimit-cofinal functor from a finite \category.
\end{example}

\begin{definition}
	Let $ K $ be a simplicial set.
	The \defn{\category presented by $ K $} is the image of $ K $ under the natural functor $ \fromto{\sSet}{\Catinfty} $ obtained by inverting the weak equivalences in the Joyal model structure.
	The \defn{space presented by $ K $} is the image of $ K $ under the natural functor $ \fromto{\sSet}{\Spc} $ obtained by inverting the weak equivalences in the Kan--Quillen model structure.
\end{definition}

\begin{example}\label{ex:simplicial_sets_with_finitely_many_simplices_in_each_dimension_are_almost_finite}
	Let $ K $ be a simplicial set with finitely many simplices in each dimension and let $ \Ical $ be the \category presented by $ K $.
	Then $ \Ical $ is almost finite: we take $ \Ical_{n} $ to be the \category presented by the $ (n+1) $-skeleton of $ \sk_{n+1} K $ and $ c_n \colon \fromto{\Ical_{n}}{\Ical} $ the functor induced by the inclusion $ \sk_{n+1} K \subset K $.
	% See \cite{MO:306596}.
\end{example}

\begin{recollection}
	A space $ X $ is \textit{almost \pifinite} if $ \uppi_0(X) $ is finite and all homotopy groups of $ X $ are finite.
	An almost \pifinite space admits a presentation by a Kan complex with finitely many simplices in each dimension \SAG{Lemma}{E.1.6.5}. 
	Hence:
\end{recollection}

\begin{example}\label{ex:amost_pi-finite_spaces_are_almost_finite_categories}
	As a special case of \Cref{ex:simplicial_sets_with_finitely_many_simplices_in_each_dimension_are_almost_finite}, every amost \pifinite space is an almost finite \category.
\end{example}

\begin{definition}
	Let $ \Ccal $ be \acategory with pullbacks.
	We say that \defn{almost finite colimits are universal in $ \Ccal $} if for each almost finite \category $ \Ical $, the \category $ \Ccal $ admits $ \Ical $-shaped colimits and $ \Ical $-shaped colimits are universal in $ \Ccal $.
\end{definition}

The argument that Lurie gives in the proof of \SAG{Theorem}{E.6.3.1} essentially shows that almost finite colimits are universal in $ \ProSpcSigma $.
However, the statement about universality of colimits is less general and the result is only stated when $ \Sigma $ is the set of all primes.
We also need to know that almost finite colimits are also universal in $ \ProSpctrun $.
The strategy is the same as Lurie's proof: we use that equivalences are checked on truncations to reduce to the case of finite colimits.

\begin{proposition}\label{prop:almost_finite_colimits_in_protruncated_spaces_are_universal}
	Almost finite colimits are universal in $ \ProSpctrun $.
\end{proposition}

\noindent Moreover, \Cref{prop:almost_finite_colimits_in_protruncated_spaces_are_universal} reproves and strengthens \SAG{Theorem}{E.6.3.1}:

\begin{corollary}\label{cor:almost_finite_colimits_in_pro-L-finite_spaces_are_universal}
	Let $ \Sigma $ be a set of primes.
	Then almost finite colimits are universal in $ \ProSpcSigma $.
\end{corollary}

\begin{proof}[Proof of \Cref{cor:almost_finite_colimits_in_pro-L-finite_spaces_are_universal}]
	Since the localization $ (-)\Sigmacomp \colon \fromto{\ProSpctrun}{\ProSpcSigma} $ is locally cartesian \cite[Proposition 3.18]{arXiv:2209.03476}, this follows from \Cref{lem:locally_cartesian_localizations_preserve_universality_of_colimits,prop:almost_finite_colimits_in_protruncated_spaces_are_universal}.
\end{proof}

\begin{proof}[Proof of \Cref{prop:almost_finite_colimits_in_protruncated_spaces_are_universal}]
	Let $ \Ical $ be an almost finite \category, let $ f \colon \fromto{X}{Z} $ be a morphism in $ \ProSpctrun $, and let
	\begin{equation*}
		g \colon \fromto{\Ical}{\ProSpctrun_{/Z}} 
	\end{equation*}
	be a diagram of protruncated spaces over $ Z $.
	To prove the claim, it suffices to show that for each integer $ n \geq 0 $, the induced map
	\begin{equation*}
		\trun_{\leq n}\paren{\colim_{i \in \Ical} X \crosslimits_{Z} g(i)} \to \trun_{\leq n}\paren{X \crosslimits_{Z} \colim_{i \in \Ical} g(i)}
	\end{equation*}
	is an equivalence in $ \Pro(\Spc_{\leq n}) $.
	Since $ \Ical $ is almost finite, there exists a finite \category $ \Ical_{n+2} $ and $ (n+2) $-colimit-cofinal functor $ c_{n+2} \colon \fromto{\Ical_{n+2}}{\Ical} $.
	Consider the commutative diagram
	\begin{equation*}
		\begin{tikzcd}
			\displaystyle \colim_{j \in \Ical_{n+2}} \trun_{\leq n}\paren{X \crosslimits_Z gc_{n+2}(j)} \arrow[r, "\sim"{yshift=-0.25em}] \arrow[d] & \displaystyle\trun_{\leq n} \biggl(\colim_{j \in \Ical_{n+2}} X \crosslimits_{Z} gc_{n+2}(j) \biggr) \arrow[r] \arrow[d] & \displaystyle\trun_{\leq n}\biggl(X \crosslimits_{Z} \colim_{j \in \Ical_{n+2}} gc_{n+2}(j) \biggr) \arrow[d] \\
			\displaystyle \colim_{i \in \Ical} \trun_{\leq n}\paren{X \crosslimits_Z g(i)} \arrow[r, "\sim"{yshift=-0.25em}] & \displaystyle\trun_{\leq n}\paren{\colim_{i \in \Ical} X \crosslimits_{Z} g(i)} \arrow[r] & \displaystyle\trun_{\leq n}\paren{X \crosslimits_{Z} \colim_{i \in \Ical} g(i)} \period
		\end{tikzcd}
	\end{equation*}
	(Here, the colimits in the leftmost column are computed in $ \Pro(\Spc_{\leq n})$.)
	Since $ \Pro(\Spc_{\leq n}) $ is an $ (n+1) $-category and $ c_{n+2} \colon \fromto{\Ical_{n+2}}{\Ical} $ is $ (n+2) $-colimit-cofinal, the leftmost vertical map is an equivalence.
	Thus the central vertical map is also an equivalence.
	Since $ \Ical_{n+2} $ is finite and finite colimits are universal in $ \ProSpctrun $ (\Cref{ex:finite_colimits_in_ProSpctrun_are_universal}), the top right-hand horizontal map is an equivalence.

	To complete the proof, it suffices to show that the rightmost vertical map is an equivalence.
	For this, consider the commutative square
	\begin{equation*}
		\begin{tikzcd}
			\displaystyle\colim_{j \in \Ical_{n+2}} \trun_{\leq n+1} gc_{n+2}(j) \arrow[r, "\sim"{yshift=-0.25em}] \arrow[d] & \displaystyle\trun_{\leq n+1} \biggl(\colim_{j \in \Ical_{n+2}}gc_{n+2}(j) \biggr) \arrow[d] \\
			\displaystyle\colim_{i \in \Ical} \trun_{\leq n+1}g(i) \arrow[r, "\sim"{yshift=-0.25em}] & \displaystyle\trun_{\leq n+1}\paren{\colim_{i \in \Ical} g(i)} \comma
		\end{tikzcd}
	\end{equation*}
	where the colimits in the left-hand column are computed in $ \Pro(\Spc_{\leq n+1}) $.
	Since $ \Pro(\Spc_{\leq n+1}) $ is an $ (n+2) $-category and $ c_{n+2} $ is $ (n+2) $-colimit-cofinal, the left-hand vertical map is an equivalence.
	Hence the right-hand vertical map is also an equivalence.
	As a consequence, the map 
	\begin{equation*}
		\colim_{j \in \Ical_{n+2}} gc_{n+2}(j) \to \colim_{i \in \Ical} g(i)
	\end{equation*}
	is $ n $-connected.
	Thus the basechange
	\begin{equation*}
		X \crosslimits_Z \colim_{j \in \Ical_{n+2}} gc_{n+2}(j) \to X \crosslimits_Z \colim_{i \in \Ical} g(i)
	\end{equation*}
	is also $ n $-connected.
	Hence the map
	\begin{equation*}
		\trun_{\leq n}\paren{X \crosslimits_{Z} \colim_{j \in \Ical_{n+2}} gc_{n+2}(j)} \to \trun_{\leq n}\paren{X \crosslimits_{Z} \colim_{i \in \Ical} g(i)}
	\end{equation*}
	is an equivalence, as desired.
\end{proof}

Note that the universality of geometric realizations implies that geometric realizations preserve finite products:

\begin{lemma}\label{lem:universal_sifted_colimits_preserve_finite_products}
	Let $ \Ical $ be a sifted \category and let $ \Ccal $ be \acategory with finite limits and $ \Ical $-shaped colimits.
	If $ \Ical $-shaped colimits are universal in $ \Ccal $, then the functor
	\begin{equation*}
		\textstyle \colim_{\Ical} \colon \fromto{\Fun(\Ical,\Ccal)}{\Ccal}
	\end{equation*}
	preserves finite products.
\end{lemma}

\begin{proof}
	Let $ \Xbullet, \Ybullet \colon \fromto{\Ical}{\Ccal} $ be functors.
	We have natural equivalences
	\begin{align*}
		\colim_{i \in \Ical} X_i \cross Y_i &\equivalence \colim_{(i,j) \in \Ical \cross \Ical} X_i \cross Y_j &&\text{($ \Ical $ is sifted}) \\
		&\equivalent \colim_{i \in \Ical} \colim_{j \in \Ical} \paren{X_i \cross Y_j} \\
		&\equivalence \colim_{i \in \Ical} \paren{X_i \cross \colim_{j \in \Ical} Y_j} &&\text{($ \Ical $-shaped colimits are universal}) \\
		&\equivalence \paren{\colim_{i \in \Ical \phantom{j}\! } X_{i}} \cross \paren{\colim_{j \in \Ical} Y_j} &&\text{($ \Ical $-shaped colimits are universal}) \period \qedhere
	\end{align*}
\end{proof}

\begin{corollary}\label{cor:geometric_realizations_preserve_finite_products_protruncated_profinite}
	Let $ \Sigma $ be a set of primes.
	Then geometric realizations preserve finite products in the \categories $ \ProSpcSigma $ and $ \ProSpctrun $.
\end{corollary}

\begin{proof}
	Combine \Cref{prop:almost_finite_colimits_in_protruncated_spaces_are_universal,cor:almost_finite_colimits_in_pro-L-finite_spaces_are_universal,lem:universal_sifted_colimits_preserve_finite_products}.
\end{proof}

% The following localizations preserve finite products:
	% \begin{enumerate}[label=\stlabel{cor:geometric_realizations_preserve_finite_products_protruncated_profinite}]
	% 	\item The geometric realization functors
	% 	\begin{equation*}
	% 		\fromto{\Fun(\Deltaop,\ProSpctrun)}{\ProSpctrun} \andeq \fromto{\Fun(\Deltaop,\ProSpcfin)}{\ProSpcfin} \period
	% 	\end{equation*}

	% 	\item The functors
	% 	\begin{equation*}
	% 		\Btrun \colon \fromto{\Catob(\ProSpctrun)}{\ProSpctrun} \andeq \Bprofin \colon \fromto{\Catob(\ProSpcfin)}{\ProSpcfin} \period
	% 	\end{equation*}

	% 	\item The functors
	% 	\begin{equation*}
	% 		\Btrun \colon \fromto{\Pro(\Cattrun)}{\ProSpctrun} \andeq \Bprofin \colon \fromto{\ProCatfin)}{\ProSpcfin} \period
	% 	\end{equation*}
	% \end{enumerate}

%-------------------------------------------------------------------%
%-------------------------------------------------------------------%
%  Completions of products                                          %
%-------------------------------------------------------------------%
%-------------------------------------------------------------------%

\section{Completions of products}\label{sec:completions_of_products}

The goal of this section is to prove \Cref{intro_prop:Sigma-completion_preserves_products,intro_prop:completions_of_products_of_etale_homotopy_types}.
We do this by noting that in the setting of both of these results, the prospaces of interest can be written as geometric realizations of simplicial profinite spaces.
We begin by axiomatizing the situation:

% \begin{notation}
% 	Let $ \Ccal $ be \acategory and let $ \Dcal \subset \Ccal $ be a full subcategory. 
% 	Given a simplicial object $ \Xbullet \colon \fromto{\Deltaop}{\Dcal} $, write 
% 	% \begin{equation*}
% 	% 	\realD{\Xbullet} \colonequals \colim\left\lparen\!
% 	% 	\begin{tikzcd}
% 	% 		\Deltaop \arrow[r, "\Xbullet"] & \Dcal
% 	% 	\end{tikzcd}
% 	% 	\!\right\rparen
% 	% \end{equation*} 
% 	% for the geometric realization of $ \Xbullet $ computed in $ \Dcal $.
% 	% Write
% 	\begin{equation*}
% 		\realC{\Xbullet} \colonequals \colim\left\lparen\!
% 		\begin{tikzcd}
% 			\Deltaop \arrow[r, "\Xbullet"] & \ProSpcfin \arrow[r, hooked] & \Ccal
% 		\end{tikzcd}
% 		\!\right\rparen
% 	\end{equation*} 
% 	for the geometric realization of $ \Xbullet $ computed in $ \Ccal $.
% 	% Importantly, if the inclusion $ \Dcal \subset \Ccal $ admits a left adjoint $ L \colon \fromto{\Ccal}{\Dcal} $, then
% 	% \begin{equation*}
% 	% 	L\realC{\Xbullet} \equivalent \realD{\Xbullet} \period
% 	% \end{equation*}
% \end{notation}

\begin{definition}\label{def:D-resolution}
	Let $ \Ccal $ be \acategory, let $ \Dcal \subset \Ccal $ be a full subcategory, and let $ X \in \Ccal $.
	A \defn{$ \Dcal $-resolution} of $ X $ is a simplicial object $ \Xbullet \colon \fromto{\Deltaop}{\Dcal} $ together with an equivalence
	\begin{equation*}
		X \equivalent \colim\left\lparen\!
		\begin{tikzcd}
			\Deltaop \arrow[r, "\Xbullet"] & \Dcal \arrow[r, hooked] & \Ccal
		\end{tikzcd}
		\!\right\rparen \period
	\end{equation*}
	We say that an object $ X $ \defn{admits a $ \Dcal $-resolution} if there exists $ \Dcal $-resolution of $ X $.
\end{definition}

\begin{notation}
	Write $ \Setfin $ for the category of a finite sets and all maps.
\end{notation}

Finite and almost \pifinite spaces admit $ \Setfin $-resolutions:

\begin{lemma}\label{lem:finite_spaces_and_almost_pi-finite_spaces_admit_Setfin-resolutions}
	Let $ X $ be a space which is finite%
	\footnote{I.e., in the smallest subcategory of $ \Spc $ containing the point and the empty set and closed under pushouts.
	Equivalently, a space $ X $ is finite if and only if $ X $ is represented by a finite CW complex.}
	or almost \pifinite.
	Then:
	\begin{enumerate}
		\item As an object of $ \Spc $, the space $ X $ admits a $ \Setfin $-resolution.

		\item When regarded as a constant object of $ \ProSpc $, the space $ X $ admits a $ \Setfin $-resolution.

		\item The protruncated space $ \protrun(X) $ admits a $ \Setfin $-resolution.

		\item The profinite space $ X\profincomp $ admits a $ \Setfin $-resolution.
	\end{enumerate}
\end{lemma}

\begin{proof}
	For (1), if $ X $ is finite, then $ X $ can be written as the geometric realization of a simplicial set with finitely many nondegenerate simplices \cite[Proposition 2.4]{arXiv:2206.02728}.%
	\footnote{In fact, every finite space is equivalent to the classifying space of a finite poset; see \cites[\href{http://www.math.ias.edu/~lurie/papers/HTT.pdf\#theorem.4.1.1.3}{Proposition 4.1.1.3(3)} \& \HTTthm{Variant}{4.2.3.16}]{HTT}[\kerodontag{02MU}]{Kerodon}[Theorem 1]{MR196744}.}
	If $ X $ is almost \pifinite, then $ X $ can be written as the geometric realization of a Kan complex with finitely many simplices in each dimension \SAG{Lemma}{E.1.6.5}.

	Now note that (2)--(4) follow from (1) and that each functor in the diagram
	\begin{equation*}
		\begin{tikzcd}
			\Spc \arrow[r, hooked] & \ProSpc \arrow[r, "\protrun"] & \ProSpctrun \arrow[r, "(-)\profincomp"] & \ProSpcfin 
		\end{tikzcd}
	\end{equation*}
	preserves colimits.
\end{proof}

Algebraic geometry also gives rise to many examples of protruncated spaces admitting profinite resolutions.

\begin{notation}[(shapes)]
	Given \atopos $ \X $, we write $ \Shape(\X) \in \ProSpc $ for the \defn{shape} of $ \X $.
	We write $ \Shapeprotrun(\X) $ for the protruncation of $ \Shape(\X) $.
\end{notation}

\begin{nul}
	Recall that \atopos $ \X $ is \textit{spectral} in the sense of \cite[Definition 9.2.1]{arXiv:1807.03281} if $ \X $ is bounded coherent and the \category $ \Pt(\X) $ of points of $ \X $ has the property that every endomorphism of an object of $ \Pt(\X) $ is an equivalence.
	The most important example of a spectral \topos is the étale \topos of a qcqs scheme \cite[Example 9.2.4]{arXiv:1807.03281}.
	\Cref{ex:protruncated_shapes_of_spectral_topoi_admit_profinite_resolutions} explains why the protruncated shape $ \Shapeprotrun(\X) $ of a spectral \topos admits a natural $ \ProSpcfin $-resolution.
\end{nul}

The next few results are the key categorical input we need.
To state them, we axiomatize the abstract properties of the subcategory $ \ProSpcfin \subset \ProSpctrun $.

\begin{notation}\label{ntn:basic_assumptions}
	Let $ \Ccal $ be \acategory with geometric realizations and finite products, and let $ L \colon \fromto{\Ccal}{\Dcal} $ be a localization.
	Assume that geometric realizations preserve finite products in both $ \Ccal $ and $ \Dcal $.
	Write
	\begin{equation*}
		\Ccal_{\abs{\Dcal}} \subset \Ccal
	\end{equation*}
	for the smallest full subcategory containing $ \Dcal \subset \Ccal $ and closed under geometric realizations and retracts.
\end{notation}

\begin{lemma}\label{lem:localization_preserves_products_of_objects_with_resolutions}
	In the setting of \Cref{ntn:basic_assumptions}, let $ \Xbullet, \Ybullet \colon \fromto{\Deltaop}{\Ccal} $ be simplicial objects with colimits $ X $ and $ Y $.
	Assume that for each $ n \geq 0 $, the natural map
	\begin{equation*}
		\fromto{L(X_n \cross Y_n)}{L(X_n) \cross L(Y_n)}
	\end{equation*} 
	is an equivalence.
	Then the natural map
	\begin{equation*}
		\fromto{L(X \cross Y)}{L(X) \cross L(Y)}
	\end{equation*} 
	is an equivalence.
\end{lemma}

\begin{proof}
	We compute
	\begin{align*}
		L(X \cross Y) &\equivalent L \paren{ \colim_{m \in \Deltaop} X_m \cross \colim_{n \in \Deltaop} Y_n } \\ 
		&\equivalent L\paren{ \colim_{n \in \Deltaop} X_n \cross Y_n }
		&& \text{(geometric realizations preserve finite products in $\Ccal$)}
		\\ 
		&\equivalent \colim_{n \in \Deltaop} L(X_n \cross Y_n) \\ 
		&\equivalence \colim_{n \in \Deltaop} L(X_n) \cross L(Y_n) &&
		\text{(assumption)} \\ 
		&\equivalent \colim_{m \in \Deltaop} L(X_m) \cross \colim_{n \in \Deltaop} L(Y_n)
		&& \text{(geometric realizations preserve finite products in $\Dcal$)} 
		\\ 
		% &\equivalent L \paren{ \colim_{m \in \Deltaop} X_m} \cross L\paren{\colim_{n \in \Deltaop} Y_n } \\
		&\equivalent L(X) \cross L(Y) \period && \qedhere
	\end{align*}
\end{proof}

\begin{corollary}\label{cor:localization_preserves_products_of_objects_with_iterated_resolutions}
	In the setting of \Cref{ntn:basic_assumptions}:
	\begin{enumerate}
		\item The full subcategory of $ \Ccal \cross \Ccal $ spanned by those objects $ (X,Y) $ such that the natural map
		\begin{equation*}
			\fromto{L(X \cross Y)}{L(X) \cross L(Y)}
		\end{equation*}
		is an equivalence is closed under geometric realizations and retracts.

		\item If $ X,Y \in \Ccal_{\abs{\Dcal}} $, then the natural map $ \fromto{L(X \cross Y)}{L(X) \cross L(Y)} $ is an equivalence.

		\item If $ X,Y \in \Ccal $ admit $ \Dcal $-resolutions, then the natural map $ \fromto{L(X \cross Y)}{L(X) \cross L(Y)} $ is an equivalence.
	\end{enumerate}
\end{corollary}

\begin{proof}
	Item (1) follows from \Cref{lem:localization_preserves_products_of_objects_with_resolutions} and the fact that equivalences are closed under retracts.
	Now (2) is an immediate consequence of (1) and the definition of $ \Ccal_{\abs{\Dcal}} $ as the closure of $ \Dcal \subset \Ccal $ under geometric realizations and retracts.
	Finally, (3) is a special case of (2).
\end{proof}

We now record some consequences of \Cref{cor:localization_preserves_products_of_objects_with_iterated_resolutions}.

\begin{corollary}\label{cor:Sigma-completion_preserves_products}
	Let $ \Sigma $ be a set of prime numbers.
	Then the $ \Sigma $-completion functor
	\begin{equation*}
		(-)\Sigmacomp \colon \fromto{\ProSpcfin}{\ProSpcSigma}
	\end{equation*}
	preserves products.
\end{corollary}

\begin{proof}
	Since $ \Sigma $-completion preserves cofiltered limits and the terminal object, it suffices to show that $ \Sigma $-completion preserves binary products of profinite spaces.
	Again because $ \Sigma $-completion preserves cofiltered limits, we are reduced to showing that if $ X $ and $ Y $ are \pifinite spaces, then the natural map
	\begin{equation*}
		\fromto{(X \cross Y)\Sigmacomp}{X\Sigmacomp \cross Y\Sigmacomp}
	\end{equation*}
	is an equivalence.
	Since \pifinite spaces admit $ \Setfin $-resolutions (\Cref{lem:finite_spaces_and_almost_pi-finite_spaces_admit_Setfin-resolutions}) and geometric realizations preserve finite products in $ \ProSpcfin $ and $ \ProSpcSigma $ (\Cref{cor:geometric_realizations_preserve_finite_products_protruncated_profinite}), the claim follows from \Cref{cor:localization_preserves_products_of_objects_with_iterated_resolutions}.
\end{proof}

\begin{warning}
	The functor $ (-)\Sigmacomp \colon \fromto{\ProSpcfin}{\ProSpcSigma} $ does not generally preserve pullbacks, or even loop objects.
\end{warning}

% \begin{counterexample}
% 	The functor $ (-)\Sigmacomp \colon \fromto{\ProSpcfin}{\ProSpcSigma} $ does not generally preserve pullbacks, or even loop objects.
% 	To see this, consider the profinite classifying space $ \BZZhat $ of the profinite group
% 	\begin{equation*}
% 		\ZZhat \isomorphic  \period 
% 	\end{equation*}
% 	Since the underlying set of $ \ZZhat $ is profinite, we have
% 	\begin{equation*}
% 		(\Omega\BZZhat)\Sigmacomp \equivalent \ZZhat \period
% 	\end{equation*} 
% 	On the other hand,
% 	\begin{equation*}
% 		(\BZZhat)\Sigmacomp \equivalent \Bup(\prod_{\el \in \Sigma} \ZZ_{\el}) \period
% 	\end{equation*}
% 	Hence 
% 	\begin{equation*}
% 		\Omega((\BZZhat)\Sigmacomp) \equivalent \prod_{\el \in \Sigma} \ZZ_{\el} \period
% 	\end{equation*}
% 	Moreover, the natural map 
% 	\begin{equation*}
% 		\fromto{(\BZZhat)\Sigmacomp}{\Omega((\BZZhat)\Sigmacomp)}
% 	\end{equation*}
% 	corresponds to the projection
% 	\begin{equation*}
% 		\fromto{\prod_{p \text{ prime}} \ZZ_p}{\prod_{\el \in \Sigma} \ZZ_{\el}} \period
% 	\end{equation*}
% \end{counterexample}

We now introduce a slight enlargement of the subcategory of protruncated spaces admitting a $ \ProSpcfin $-resolution on which profinite completion preserves finite products.

\begin{notation}
	Let
	\begin{equation*}
		\ProSpctrun' \subset \ProSpctrun
	\end{equation*}
	denote the smallest full subcategory containing $ \ProSpcfin $ and closed under geometric realizations, retracts, and cofiltered limits.
\end{notation}

\begin{observation}[(procompact spaces)]
	Write $ \Spc^{\upomega} \subset \Spc $ for the full subcategory spanned by the compact objects.%
	\footnote{A space $ X $ is compact if and only if $ X $ is a retract of a finite space.
	In more classical terminology, a space $ X $ is compact if and only if $ X $ is represented by a \textit{fintiely dominated} CW complex.}
	Then by \Cref{lem:finite_spaces_and_almost_pi-finite_spaces_admit_Setfin-resolutions}, the image of 
	\begin{equation*}
		\protrun \colon \fromto{\Pro(\Spc^{\upomega})}{\ProSpctrun}
	\end{equation*}
	is contained in $ \ProSpctrun' $.
\end{observation}

The following is the main result of this note.

\begin{theorem}\label{thm:profinite_completion_preserves_products_admitting_resolutions_by_simplicial_profinite_spaces}
	Let $ \Sigma $ be a set of primes and let $ X, Y \in \ProSpctrun' $.
	Then the natural map
	\begin{equation*}
		\fromto{(X \cross Y)\Sigmacomp}{X\Sigmacomp \cross Y\Sigmacomp} 
	\end{equation*}
	is an equivalence.
\end{theorem}

\begin{proof}
	By \Cref{cor:Sigma-completion_preserves_products}, the $ \Sigma $-completion functor
	\begin{equation*}
		(-)\Sigmacomp \colon \fromto{\ProSpcfin}{\ProSpcSigma}
	\end{equation*}
	preserves products.
	Hence it suffices to prove the claim in the special case where $ \Sigma $ is the set of all primes.
	For this, note that cofiltered limits preserve finite products in $ \ProSpctrun $ and $ \ProSpcfin $; moreover, the profinite completion functor
	\begin{equation*}
		(-)\profincomp \colon \fromto{\ProSpctrun}{\ProSpcfin}
	\end{equation*}
	preserves cofiltered limits.
	Thus it suffices to treat the case where $ X $ and $ Y $ are in the smallest full subcategory of $ \ProSpctrun $ containing $ \ProSpcfin $ and closed under geometric realizations and retracts.
	In this case, since geometric realizations preserve finite products in $ \ProSpctrun $ and $ \ProSpcfin $ (\Cref{cor:geometric_realizations_preserve_finite_products_protruncated_profinite}), \Cref{cor:localization_preserves_products_of_objects_with_iterated_resolutions} completes the proof.
\end{proof}

\begin{example}
	If $ X $ and $ Y $ are protruncated spaces that admit $ \ProSpcfin $-resolutions, then the natural map $ \fromto{(X \cross Y)\Sigmacomp}{X\Sigmacomp \cross Y\Sigmacomp} $ is an equivalence.
\end{example}

\begin{example}
	If $ X,Y \in \Spc $ are compact, then the natural map $ \fromto{(X \cross Y)\Sigmacomp}{X\Sigmacomp \cross Y\Sigmacomp} $ is an equivalence.
\end{example}

We conclude by recording two applications of \Cref{thm:profinite_completion_preserves_products_admitting_resolutions_by_simplicial_profinite_spaces}.

\begin{corollary}\label{cor:profinite_completion_preserves_products_of_shapes_of_spectral_topoi}
	Let $ \Sigma $ be a set of primes and let $ \X $ and $ \Y $ be spectral \topoi.
	Then the natural map
	\begin{equation*}
		\fromto{\paren{\Shape(\X) \cross \Shape(\Y)}\Sigmacomp}{\Shape(\X)\Sigmacomp \cross \Shape(\Y)\Sigmacomp} 
	\end{equation*}
	is an equivalence.
\end{corollary}

\begin{proof}
	\Cref{ex:protruncated_shapes_of_spectral_topoi_admit_profinite_resolutions} shows that $ \Shapetrun(\X) $ and $ \Shapetrun(\Y) $ admit $ \ProSpcfin $-resolutions.
	Since protruncation preserves products, the claim now follows from \Cref{thm:profinite_completion_preserves_products_admitting_resolutions_by_simplicial_profinite_spaces}.
\end{proof}

\begin{example}\label{ex:profinite_completion_preserves_products_of_etale_homotopy_types}
	Let $ \Sigma $ be a set of primes and let $ X $ and $ Y $ be qcqs schemes.
	Then the natural map
	\begin{equation*}
		\fromto{\paren{\Piet(X) \cross \Piet(Y)}\Sigmacomp}{\Piet(X)\Sigmacomp \cross \Piet(Y)\Sigmacomp} 
	\end{equation*}
	is an equivalence.
\end{example}

The second application is to \operads.
Recently, it has become clear that it is important to consider profinite completions of \operads.
For example, Boavida de Brito, Horel, and Robertson explained a beautiful relationship between the Grothendieck--Teichmüller group and the profinite completion of the genus $ 0 $ surface operad \cite{MR3921321}.
However, since profinite completion does not preserve products, profinite completions of \operads generally cannot be computed `levelwise'.
In general, more care is required to work with profinite completions of \operads; this is the topic of recent work of Blom and Moerdijk \cites{MR4462939}{arXiv:2312.12567}.

An immediate consequence of \Cref{thm:profinite_completion_preserves_products_admitting_resolutions_by_simplicial_profinite_spaces} is that profinite completions \textit{can} be computed levelwise more generally than was previously known:

\begin{corollary}\label{cor:profinite_completions_of_operads}
	Let $ \Sigma $ be a set of primes and let $ \{\Ocal(n)\}_{n \geq 0} $ be an \operad in spaces.
	Assume that for each $ n \geq 0 $, we have
	\begin{equation*}
		\protrun \Ocal(n) \in \ProSpctrun' \period
	\end{equation*}
	Then the levelwise \Sigmacompletion $ \{\Ocal(n)\Sigmacomp \}_{n \geq 0} $ defines an \operad in $ \ProSpcSigma $.
\end{corollary}
	
%-------------------------------------------------------------------%
%-------------------------------------------------------------------%
%  Classifying prospaces via geometric realizations                 %
%-------------------------------------------------------------------%
%-------------------------------------------------------------------%

\appendix

\section{Classifying prospaces via geometric realizations}\label{app:classifying_prospaces_via_geometric realizations}

The purpose of this appendix is to explain why the protruncated shape of a spectral \topos (e.g., the étale \topos of a qcqs scheme) admits a presentation as a geometric realization of a simplicial profinite space (\Cref{ex:protruncated_shapes_of_spectral_topoi_admit_profinite_resolutions}). 
Using the description of the protruncated shape of a spectral \topos as a protruncated classifying prospace given by Barwick--Glasman--Haine \cite[Theorem 10.2.3]{arXiv:1807.03281}, this is an exercise in the definitions.
For the ease of the reader, we spell out the details here.

We make use of the description of \categories as simplicial spaces.

\begin{recollection}[{(\categories as simplicial spaces)}]\label{rec:complete_Segal_spaces}
	The nerve functor
	\begin{align*}
		\Nerve \colon \Catinfty &\to \Fun(\Deltaop,\Spc) \\ 
		\Ccal &\mapsto \brackets{I \mapsto \Fun(I,\Ccal)^{\equivalent}}
	\end{align*}
	is fully faithful \cites[\HAappthm{Proposition}{A.7.10}]{HA}[\SAGsubsec{A.8.2}]{SAG}{MR2342834}[\S1]{Lurie:Goodwillie-I}{MR1804411}.
	One can explicitly identify its image; objects in the image of this embedding are often called \textit{complete Segal spaces} or \textit{categories internal to spaces}.
	Under this embedding, the subcategory $ \Spc \subset \Catinfty $ corresponds to the constant functors $ \fromto{\Deltaop}{\Spc} $.
	Moreover, the localization $ \Bup \colon \fromto{\Catinfty}{\Spc} $ is given by geometric realization.
\end{recollection}

\begin{notation}
	Let $ \Ccal $ be \acategory with finite limits.
	We write 
	\begin{equation*}
		\Catob(\Ccal) \subset \Fun(\Deltaop,\Ccal)
	\end{equation*}
	for the full subcategory spanned by the \textit{categories internal to $ \Ccal $}.
	See \cites[Definition 13.1.1]{arXiv:1807.03281}[Proposition 3.2.7]{arXiv:2103.17141} for the definition.
\end{notation}

\begin{notation}
	Write $ \Cattrun \subset \Catinfty $ for the full subcategory spanned by those \categories $ \Ccal $ for which there exists an integer $ n \geq 0 $ such that $ \Ccal $ is an $ n $-category. 
	Write $ \Catfin \subset \Cattrun $ for the full subcategory spanned by those \categories with the property that there are finitely many objects up to equivalence and all mapping spaces are \pifinite.
\end{notation}

\begin{observation}
	The nerve $ \Nerve \colon \equivto{\Catinfty}{\Catob(\Spc)} $ restricts to equivalences
	\begin{equation*}
		\equivto{\Cattrun}{\Catob(\Spctrun)} \andeq \equivto{\Catfin}{\Catob(\Spcfin)} \period
	\end{equation*}
\end{observation}

In order to describe protruncated classifying spaces via geometric realizations, it is useful to describe \procategories as category objects in prospaces.

\begin{observation}
	By \cites[\HTTthm{Proposition}{5.3.5.11}]{HTT}[Proposition 13.1.12]{arXiv:1807.03281}, the composite
	\begin{equation*}
		\begin{tikzcd}
			\Cattrun \arrow[r, "\Nerve"', "\sim"{yshift=-0.25em}] & \Catob(\Spctrun) \arrow[r, hooked] & \Catob(\ProSpctrun)
		\end{tikzcd}
	\end{equation*}
	extends along cofiltered limits to a fully faithful right adjoint
	\begin{equation*}
		\Nerve \colon \incto{\Pro(\Cattrun)}{\Catob(\ProSpctrun)} \period
	\end{equation*}
	This functor restricts to a fully faithful right adjoint 
	\begin{equation*}
		\Nerve \colon \incto{\Pro(\Catfin)}{\Catob(\ProSpcfin)} \period
	\end{equation*}
\end{observation}

\begin{remark}
	We do not know if the embedding $ \Nerve \colon \incto{\Catinfty}{\Cat(\ProSpc)} $ extends along cofiltered limits to a fully faithful functor $ \fromto{\Pro(\Catinfty)}{\Catob(\ProSpc)} $.
\end{remark}

% \begin{notation}
% 	Let $ \X $ be a spectral \topos.
% 	Write
% 	\begin{equation*}
% 		\Pt_{\bullet}(\X) \colon \fromto{\Deltaop}{\ProSpcfin}
% 	\end{equation*}
% 	for the category object in profinite spaces corresponding to the profinite stratified shape of $ \X $ under the embedding
% 	\begin{equation*}
% 		\ProCatfin \subset \Catob(\ProSpcfin) \period
% 	\end{equation*}
% \end{notation}

\begin{observation}
	It is immediate from the definitions that the following diagram of fully faithful right adjoints commutes
	\begin{equation*}
		\begin{tikzcd}[sep=3em]
			\ProSpcfin \arrow[d, phantom, "\scriptscriptstyle\shortlefttack" description] \arrow[r, phantom, "\scriptscriptstyle\shortuptack" description] \arrow[r, hooked, shift right] \arrow[d, hooked, shift left] & %
			\ProCatfin \arrow[d, phantom, "\scriptscriptstyle\shortlefttack" description] \arrow[r, phantom, "\scriptscriptstyle\shortuptack" description] \arrow[l, "\Bprofin"', shift right] \arrow[r, hooked, "\Nerve"', shift right] \arrow[d, hooked, shift left] & %
			\Catob(\ProSpcfin) \arrow[d, phantom, "\scriptscriptstyle\shortlefttack" description] \arrow[r, phantom, "\scriptscriptstyle\shortuptack" description] \arrow[l, shift right] \arrow[r, hooked, shift right] \arrow[d, hooked, shift left] & %
			\Fun(\Deltaop, \ProSpcfin) \arrow[d, phantom, "\scriptscriptstyle\shortlefttack" description] \arrow[l, shift right]  \arrow[d, hooked, shift left] \\
			%-------------------------------------------------------------------%
			\ProSpctrun \arrow[r, phantom, "\scriptscriptstyle\shortuptack" description] \arrow[r, hooked, shift right] \arrow[u, shift left, "(-)\profincomp"] & %
			\ProCattrun \arrow[r, phantom, "\scriptscriptstyle\shortuptack" description] \arrow[l, "\Btrun"', shift right] \arrow[r, hooked, "\Nerve"', shift right] \arrow[u, shift left] & %
			\Catob(\ProSpctrun) \arrow[r, phantom, "\scriptscriptstyle\shortuptack" description] \arrow[l, shift right]  \arrow[r, hooked, shift right] \arrow[u, shift left] & %
			\Fun(\Deltaop, \ProSpctrun) \period \arrow[l, shift right] \arrow[u, shift left, "(-)\profincomp \of -"]
		\end{tikzcd}
	\end{equation*}
	The long composite left adjoints
	\begin{equation*}
		\fromto{\Fun(\Deltaop,\ProSpctrun)}{\ProSpctrun} \andeq \fromto{\Fun(\Deltaop,\ProSpcfin)}{\ProSpcfin} \period
	\end{equation*}
	are simply the colimit functors.
	Since the diagram of left adjoints also commutes we deduce:
\end{observation}

\begin{corollary}\label{cor:presentation_of_classifying_protruncated_spaces_of_profinite_categories}
	Let $ \Ccal \in \ProCatfin $ be a profinite \category.
	Then there are natural equivalences
	\begin{align*}
		\Btrun(\Ccal) &\equivalent
		\colim\left\lparen\!
		\begin{tikzcd}[sep=3em, ampersand replacement=\&]
			\Deltaop \arrow[r, "\Nerve(\Ccal)"] \& \ProSpcfin \arrow[r, hooked] \& \ProSpctrun
		\end{tikzcd}
		\!\right\rparen \\ 
		\shortintertext{and}
		\Bprofin(\Ccal) &\equivalent
		\colim\left\lparen\!
		\begin{tikzcd}[sep=3em, ampersand replacement=\&]
			\Deltaop \arrow[r, "\Nerve(\Ccal)"] \& \ProSpcfin
		\end{tikzcd}
		\!\right\rparen \period 
	\end{align*}
\end{corollary}

\begin{example}\label{ex:protruncated_shapes_of_spectral_topoi_admit_profinite_resolutions}
	Let $ \X $ be a spectral \topos.
	Through \cite[Theorem 9.3.1]{arXiv:1807.03281}, Barwick--Glasman--Haine refined the the \category $ \Pt(\X) $ of points of $ \X $ to a profinite \category
	\begin{equation*}
		\Pihat_{(\infty,1)}(\X) \in \ProCatfin
	\end{equation*}
	called the \textit{stratified shape} of $ \X $.
	In \cite[Theorem 10.2.3]{arXiv:1807.03281} they show that there is a natural equivalence
	\begin{equation*}
		\equivto{\Shapetrun(\X)}{\Btrun\paren{\Pihat_{(\infty,1)}(\X)}} \period
	\end{equation*}
	That is, the protruncated shape of $ \X $ can be recovered as the protruncated classifying space of the stratified shape $ \Pihat_{(\infty,1)}(\X)  $.
	Hence \Cref{cor:presentation_of_classifying_protruncated_spaces_of_profinite_categories} shows that the protruncated shape $ \Shapeprotrun(\X) $ admits a natural $ \ProSpcfin $-resolution in the sense of \Cref{def:D-resolution}.
\end{example}

\begin{remark}
	Using the proétale topology \cite{MR3379634}, one can show that the protruncated étale homotopy type of a qcqs scheme admits a $ \Pro(\Setfin) $-resolution.

	To see this, first note that given a qcqs scheme $ X $, the pullback functor \smash{$ \nuupperstar \colon \fromto{X_{\et}}{X_{\proet}} $} from the étale $ \infty $-topos of $ X $ to the proétale $ \infty $-topos of $ X $ is fully faithful when restricted to truncated objects.
	(This observation generalizes \cite[Lemma 5.1.2 \& Corollary 5.1.6]{MR3379634}.)
	As a result, the protruncated shapes of $ X_{\et} $ and $ X_{\proet} $ are equivalent.
	Hence the protruncated étale homotopy type is a hypercomplete proétale cosheaf.
	
	Recall that the w-contractible affine schemes form a basis for the proétale topology; in particular, every qcqs scheme admits a proétale hypercover by w-contractible affine schemes.
	Moreover, for each w-contractible affine scheme $ U $, the prospace $ \Pietprotrun(U) $ is naturally identified with the profinite set $ \uppi_0(U) $ of connected components of $ U $.
	Hence the protruncated étale homotopy type
	\begin{equation*}
		\Pietprotrun \colon \fromto{\Schqcqs}{\ProSpctrun}
	\end{equation*}
	is the unique hypercomplete proétale cosheaf whose restriction to w-contractible affines is given by $ \goesto{U}{\uppi_0(U)} $.
	Given a qcqs scheme $ X $, choose a proétale hypercover $ U_{\bullet} $ of $ X $ be w-contractible affines.
	Then the simplicial object $ \uppi_0(U_{\bullet}) $ is a $ \Pro(\Setfin) $-resolution of $ \Pietprotrun(X) $.
\end{remark}

%-------------------------------------------------------------------%
%-------------------------------------------------------------------%
%  References                                                       %
%-------------------------------------------------------------------%
%-------------------------------------------------------------------%

\DeclareFieldFormat{labelnumberwidth}{#1}
\printbibliography[keyword=alph, heading=references]
\DeclareFieldFormat{labelnumberwidth}{{#1\adddot\midsentence}}
\printbibliography[heading=none, notkeyword=alph]

\end{document}